\documentclass[reqno]{amsart}
\usepackage{amsmath, amsthm, amscd, amsfonts, amssymb, graphicx, color}
\usepackage[all]{xy}

\usepackage[colorlinks, citecolor=black, filecolor=black, linkcolor=black,urlcolor=black,]{hyperref}

\newcommand{\kk}             {{\mathbb K}}
\newcommand{\rr}             {{\mathbb R}}
\newcommand{\into}          {\rightarrow}
\newcommand{\tens}         {\otimes}
\newcommand{\fhi}            {\varphi}
\newcommand{\so}            {\Sigma\mbox{-ope\-ra\-tor}}

\newcommand{\sumim}      {\sum\limits_{i=1}^m}
\newcommand{\lev}            {\left\langle}
\newcommand{\rev}            {\right\rangle}

\newcommand{\br}    		    {{\beta|}}
\newcommand{\apq}  	    {\alpha_{p,q}}
\newcommand{\apqb}  	    {\alpha_{p,q}^\beta}

\newcommand{\xxx}           {X_1,\dots, X_n}
\newcommand{\zzz}            {Z_1,\dots, Z_n}
\newcommand{\eee}          {E_1,\dots, E_n}
\newcommand{\xpx}           {X_1\times\dots\times X_n}
\newcommand{\epe}           {E_1\times\dots\times E_n}
\newcommand{\zpz}           {Z_1\times\dots\times Z_n}
\newcommand{\xxp}           {(x^1,\dots ,x^n)}

\newcommand{\sxx}           {\Sigma_{X_1\dots X_n}}
\newcommand{\sxxb}         {\Sigma_{X_1 \dots X_n}^\beta}
\newcommand{\szzt}           {\Sigma_{Z_1 \dots Z_n}^\theta}
\newcommand{\szz}            {\Sigma_{Z_1\dots Z_n}}
\newcommand{\sxxp}         {\Sigma_{X_1 \dots X_n}^\pi}
\newcommand{\xtx}      		{X_1\tens\dots\tens X_n}
\newcommand{\ztz}      		{Z_1\tens\dots\tens Z_n}
\newcommand{\ete}      		{E_1\tens\dots\tens E_n}
\newcommand{\xtxb}    		{\left(\xtx,\beta\right)}
\newcommand{\ztzt}    		{\left(\ztz,\theta\right)}
\newcommand{\xtxty}    	  	{X_1\tens\dots\tens X_n\tens Y}
\newcommand{\etetf}    	  	{E_1\tens\dots\tens E_n\tens F}
\newcommand{\ztztw}    	{Z_1\tens\dots \tens Z_n\tens W}
\newcommand{\xtxtyd}   	{X_1\tens\dots\tens X_n\tens Y^*}
\newcommand{\xtxp}         {X_1\hat{\tens}_\pi\dots\hat{\tens}_\pi X_n}
\newcommand{\xxt}           {x^1\tens\dots\tens  x^n}
\newcommand{\zzt}            {z^1\tens\dots\tens z^n}

\newcommand{\pimqi}       {p_i-q_i}

\newcommand{\rlu}            {\sumim \lambda_i(\pimqi)\tens y_i}

\newcommand{\Txxy}         {T:\xpx\into Y}
\newcommand{\tlin}           {\widetilde{T}}
\newcommand{\rlin}           {\widetilde{R}}

\newcommand{\Lxx}           {\mathcal{L}(\xxx)}
\newcommand{\Lxxy}         {\mathcal{L}(\xxx;Y)}
\newcommand{\Lbxx}         {\mathcal{L}^\beta\left(\xxx\right)}
\newcommand{\Dpqbxxy}   {\mathcal{D}_{p,q}^\beta(\xxx; Y)}
\newcommand{\Lb}            {Lip^\beta}

\newtheorem{theorem}{Theorem}[section]
\newtheorem{lemma}[theorem]{Lemma}
\newtheorem{proposition}[theorem]{Proposition}
\newtheorem{corollary}[theorem]{Corollary}
\theoremstyle{definition}
\newtheorem{definition}[theorem]{Definition}

\theoremstyle{remark}
\newtheorem{remark}[theorem]{Remark}
\numberwithin{equation}{section}

\begin{document}

\title{$(p,q)$-Dominated Multilinear Operators and Laprest\'e tensor norms}
\author{Maite Fern\'andez-Unzueta, Samuel Garc\'ia-Hern\'andez}
\address{Centro de Investigaci\'{o}n en Matem\'{a}ticas, P.O. Box 402, Jalisco S/N Mineral de Valenciana, Guanajuato, M\'{e}xico}
\email{maite@cimat.mx; orcid:{0000-0002-8321-4877}} 
\email{samuelg@cimat.mx; orcid:{0000-0003-4562-299X}}








\begin{abstract}
We introduce a notion of $(p,q)$-dominated multilinear operators which stems from the  geometrical approach provided by  $\Sigma$-operators. We prove that  $(p,q)$-dominated multilinear operators can be characterized in terms of their  behavior on finite sequences and in terms of their  relation with a Laprest\'e tensor norm.  We also prove that they verify a generalization of the Pietsch's Domination Theorem and Kwapie\'n's Factorization Theorem. Also, we  study the collection $\mathcal{D}_{p,q}$ of all $(p,q)$-dominated multilinear operators  showing that $\mathcal{D}_{p,q}$ has a maximal ideal demeanor and that the Laprest\'e norm has a finitely generated behavior.
\end{abstract}
\maketitle

\section{Introduction}

In recent decades, a lot of research has been focused on extending the theory of linear operators to the multilinear setting. This has been the case of compact, nuclear, integral and absolutely $p$-summing operators, among others, see for instance \cite{alencar85, angulo18, dimant03, fernandez-unzueta18b, krikorian72, lopez-molina07, pellegrino14, pellegrino16, pietsch83, villanueva03}. For the case  of dominated operators,  we find the notions of  $(r_1,\dots, r_n)$-dominated multilinear operators, treated by several authors in \cite{carando06, carando07, jarchow07, matos93, melendez99, pellegrino14, pellegrino11, perez-garcia05a, popa14c}, the class of  $(r_0, r_1\dots, r_{n+1})$-dominated multilinear operators \cite{lopez-molina12}, the dominated $n$-linear operators respect to a system of $n$ linear operators \cite{popa17} and the dominated $(p_1,\dots , p_m;\sigma)$-continuous multilinear operators \cite{achour16}.

A significant feature of the class  of $(p,q)$-dominated linear operators  is that it encompasses  two  of the most important  classes  of linear operators, namely the  absolutely $p$-summing operators, (case $q=\infty$) and the   $p$-dominated operators (case $q=p^*$).
 It constitutes a  maximal ideal, which is linked with the Laprest\'e tensor norms  through a duality relation. The  name  $(p,q)$-dominated is  due to the fact that  they can be characterized in terms of  a  domination inequality  which is a   generalization of   the Pietsch domination inequality of the absolutely $p$-summing operators.  Another key feature of these operators is that they verify the   so called  Kwapie\'n's Factorization Theorem.  Excellent expositions of this class can be found in \cite[Section 17.4]{pietsch78} and \cite[Section 19]{defant93}.

The goal of this paper is to introduce and develop  a   notion of $(p,q)$-domination for  multilinear mappings $\Txxy$ such that for it,  the key  features of  $(p,q)$-dominated linear operators mentioned above remain true also in the more general setting of multilinear operators. This notion stems from regarding $T$ as a homogeneous mapping on the Segre cone $\sxx$ (a certain subset of the tensor product $\xtx$).

The results obtained here are consistent with those obtained in \cite{angulo18}, where the same geometrical generalization procedure of \cite{fernandez-unzueta18a} was applied. Concretely, $(p,q)$-dominated multilinear operators are closely related with the so called Lipschitz $p$-summing multilinear operators (see (iii) of Theorem \ref{factorization}).

It is important to note that to have a complete description of the multilinear case,  it is necessary to deal simultaneously with all reasonable crossnorms on $\xtx$.

To summarize the content, in the case of the projective tensor norm $\pi$, consider $1\leq p,q\leq \infty$ with the property $\frac{1}{p}+\frac{1}{q}\leq 1$ and let $\Txxy$ be a bounded multilinear operator between Banach spaces such that there exists a constant $C$ verifying
\begin{align}\label{multidominated}
    \|\left(\lev y_i^*, T(x_i^1,\dots, x_i^n\right.\right.&\left.)\left.-T(z_i^1,\dots, z_i^n)  \rev \right)_{i=1}^m  \|_{r^*}\nonumber \\
	                                                &\leq C\|(x_i^1\tens\dots\tens x_i^n-z_i^1\tens\dots\tens z_i^n)_{i=1}^m\|_p^{w,\pi}  \|(y_i^*)_{i=1}^m\|_q^w
\end{align}
for all finite sequences $(x_i^1,\dots, x_i^n)_{i=1}^m$, $(z_i^1,\dots, z_i^n)_{i=1}^m$ in $\xpx$ and $(y_i^*)_{i=1}^m$ in $Y^*$. Condition \eqref{multidominated} is equivalent to the boundedness of the functional
\begin{eqnarray}\label{multitensor}
\fhi_T:(\xtxtyd,\alpha_{q^*,p^*}^\pi) &\into& \kk\\
                        \xxt\tens y^*   &\mapsto& \lev  y^*  ,  T\xxp \rev,\nonumber
\end{eqnarray}
where $\alpha_{q^*,p^*}^\pi$ is a tensor norm which we call a Laprest\'e tensor norm (see Definition \ref{laprestenorms}). Moreover, \eqref{multitensor} implies the existence of a constant $S$ and probability measures $\mu$ and $\nu$ on $B_{\Lxx}$ and $B_{Y^{**}}$ such that
\begin{align}\label{multidomination}
| \lev  y^* \right.&\left., T(x) - T(z) \rev |\nonumber\\
                 &\leq S\left(\int\limits_{B_{\Lxx}} |\psi(x)-\psi(z)|^{p}d\mu(\psi)\right)^{\frac{1}{p}}     \left(\int\limits_{B_{Y^{**}}} |y^{**}(y^*)|^{q}d\nu(y^{**})\right)^{\frac{1}{q}}
\end{align}
for all $x,z$ in $\xpx$ and $y^*$ in $Y^*$. For its part,  inequality \eqref{multidomination} implies that $T$ factors as follows
\begin{small}
\begin{eqnarray}\label{multifactorization}
\begin{array}{c}
\xymatrix{
\xpx\ar[dr]_{A}\ar[rr]^-T &                         & Y \\
                                            &A(\xpx)\ar[ur]_B\ar[d]  &     \\
                                            &         Z              &
}
\end{array},
\end{eqnarray}
\end{small}
where $Z$ is a Banach space, $A:\xpx\into Z$ is a Lipschitz $p$-summing multilinear operator and $B:A(\xpx)\into Y$ is a Lipschitz function whose adjoint linear operator $B^*:Y^*\into (A(\xpx))^*$ is $q$-summing. Furthermore, this kind of factorization for a multilinear operator implies \eqref{multidominated}. Thus, the four conditions are equivalent.

Some remarks are in order. In \eqref{multidominated} $r$ is a suitable value depending on $p$ and $q$ and $\|\cdot \|_p^{w,\pi}$ is the weak $p$-norm of weak $p$-summable sequences in $\xtxp$ (see \cite[Chapter 2]{diestel95}). In \eqref{multifactorization}, the notion of Lipschitz $p$-summing multilinear operators is the one introduced in \cite{angulo18}. The details of a Lipschitz function $B$ and its linear adjoint $B^*$ are given in Lemma \ref{sigmaimages}.  If a multilinear operator $T$ verifies any of the equivalences listed above, then the best possible constant of \eqref{multidominated} coincides with the norm of $\fhi_T$ in \eqref{multitensor}, with the best constant in \eqref{multidomination} and with $\inf\pi_p(A)\pi_q(B^*)$ where the infimum is taken over all possible factorizations as in \eqref{multifactorization}.

We will say that a bounded multilinear operator $\Txxy$ between Banach spaces is $\pi$-$(p,q)$-dominated if it verifies any of the equivalences listed above. As we will see, the set of all $\pi$-$(p,q)$-dominated multilinear operators from $\xpx$ to $Y$ is a vector space and
$$D_{p,q}^\pi(T):=\inf C = \|\fhi_T\| = \inf S = \inf \pi_p(A)\pi_q(B^*)$$
defines a norm on it.

When  $n=1$,  we recover the class of $(p,q)$-dominated linear operators. As we said before, in the linear setting the properties of being $(p,\infty)$-dominated and $p$-summing are the same. In the multilinear  case, for an arbitrary value of $n$, a multilinear operator $T$ is $\pi$-$(p,\infty)$-dominated if and only if $T$ is Lipschitz $p$-summing (see Corollary \ref{psumming}). That is, the class of $\pi$-$(p,q)$-dominated operators contains the class of Lipschitz $p$-summing multilinear operators.

In Section 2 we introduce the Laprest\'e tensor norm $\alpha_{p,q}^\pi$ and characterize the functionals $\fhi$ bounded  on $(\xtxty,\alpha_{p,q}^\pi)$. In Section 3 we prove the equivalences listed above selecting condition \eqref{multidominated} as definition of a $\pi$-$(p,q)$-dominated multilinear operator.

In Subsection 4.1 we prove that the $\pi$-$(p,q)$-dominated property has a local behavior, this is, it depends on the finite dimensional subspaces of the factors $X_i$. To prove this, we have to extend our study of $\pi$-$(p,q)$-dominated multilinear operators $\Txxy$ regarding any reasonable crossnorm $\beta$ on $\xtx$ (see Definition \ref{betadominatedoperators} and Theorem \ref{maximal}). Similarly, we show that the Laprest\'e tensor norm $\alpha_{p,q}^\pi$ of a tensor $u$ in $\xtxty$ is determined by the finite dimensional subspaces of $X_i$ and $Y$. For proving this result we introduce the Laprest\'e tensor norms $\alpha_{p,q}^\beta$ (see Definition \ref{betalaprestenorms} and Proposition \ref{finitelygenerated}).

In Subsection 4.2 we study of the class $\mathcal{D}_{p,q}$ of all $(p,q)$-dominated multilinear operators as an ideal. The combination of Proposition \ref{ideal} and Theorem \ref{maximal} tells us that $\mathcal{D}_{p,q}$ has a maximal ideal demeanor. We also explore some properties of the Laprest\'e tensor norms $\alpha_{p,q}^\beta$ in Proposition \ref{tensornorm}. This result when combined with Proposition \ref{finitelygenerated} says that the Laprest\'e tensor norms $\alpha_{p,q}^\beta$ behaves alike a finitely generated tensor norm. Finally, in Theorem \ref{tensorialrepresentation} we give a complete representation of the class $\mathcal{D}_{p,q}$ in terms of the Laprest\'e tensor norms $\alpha_{p,q}^\beta$.

In Section 5 we make some final remarks about the obtained results.


\subsection{A Geometric Approach to Work with  Multilinear Mappings}
In this subsection we briefly describe the geometric approach we have used  to derive our   notion of $(p,q)$-domination for multilinear mappings.   It consists, basically, in study a multilinear map $T$ by means of an auxiliary   homogeneous  function  $f_T$ called  its associated $\Sigma$-operator. The details can be found in \cite{fernandez-unzueta18a}.

We  use standard notation of Banach spaces theory, multilinear operators and tensor products. The letter $\kk$ denotes the field of real or complex numbers. The unit ball of the normed space $X$ is denoted by $B_X$. The operator $K_X:X\into X^{**} $ denotes the canonical embedding giving by evaluation.

Throughout this work $n$ denotes a positive integer and the capital letters $X_1,\dots,$ $X_n$, $Y$ and $Z$ denote Banach spaces over the same field. The symbol $\Lxxy$ denotes the Banach space of all bounded multilinear operators $T:\xpx\into Y$ with the usual uniform norm $\|T\|=\sup\{ \|T\xxp\| : \|x_i\|\leq 1\}$. We simply write $\mathcal{L}(\xxx)$ when $Y=\kk$.

The set of decomposable tensors of the algebraic tensor product $\xtx$ is denoted by $\sxx$. That is, $\sxx:=\left\{\;  \xxt  \;|\;  x^i\in X_i  \;\right\}$. We denote by $\pi$ the projective tensor norm on $\xtx$ given by
$$\pi(u)= \inf \left\{\;  \sumim \|x_i^1\|\dots\|x_i^n\|  \;\Big|\; u=\sumim x_i^1\tens\dots\tens x_i^n \;\right\}$$
for all $u$. We denote by $\sxxp$ the Segre cone of the spaces, that is, the metric space resulting by restricting the norm $\pi$ of $\xtxp$ to the set $\sxx$.

According to the universal property of the projective tensor product, for every bounded  multilinear operator $T:\xpx\into Y$ there exists a unique bounded linear operator $\tlin:\xtxp\into Y$ such that $T\xxp=\tlin(\xxt)$ for all $x^i\in X_i$, $1\leq i \leq n$. In particular, the restriction $\tlin|_{\sxxp}:\sxxp\into Y$ is a Lipschitz function. In this situation, the operator $\tlin$ is called the linearization of $T$ and the function $f_T:=\tlin|_{\sxxp}$ is named the $\so$ associated to $T$. In \cite[Theorem 3.2]{fernandez-unzueta18a} it is proved that $\|T\|=Lip(T)=\|\tlin\|$.

For a Banach space $X$ and $1\leq p\leq \infty$ we use the standard notation $\|\cdot\|_p^w$ to denote the norm of $p$-weak summable sequences in $X$. For the case of $\xtx$ we write explicitly the reasonable crossnorm, for example, if $(p_i)_{i=1}^m$ and $(q_i)_{i=1}^m$ are two finite sequences in $\sxx$ and $p$ is finite, then
$$\|(p_i-q_i)_{i=1}^m\|_p^{w,\pi}=\sup\limits_{\fhi\in B_{\Lxx}} \left(\sumim |f_\fhi(p_i)-f_\fhi(q_i)|^p\right)^{\frac{1}{p}}.$$

We will  require the consideration of sets of the form $f_A(\sxx)=A(\xpx)$ where $A:\xpx\into Z$ is a bounded multilinear operator. We collect some important facts of these sets in the next lemma, which  was already used in \cite{fernandez-unzueta18b}.

\begin{lemma}\label{sigmaimages}
Let $A:\xpx\into Z$ be a bounded multilinear operator between Banach spaces. Then:
\begin{itemize}
\item [i)] The set $(f_A(\sxx))^*$ of all Lipschitz functions $\psi:f_A(\sxx)\into \kk$ such that $\psi A$ is multilinear is a vector space endowed with the algebraic operations defined pointwise; moreover, it becomes a Banach space with the Lipschitz norm.
\item [ii)] Let $B:f_A(\sxx)\into Y$ be a Lipschitz function such that the composition $BA:\xpx\into Y$ is multilinear. The function
\begin{eqnarray}\label{adjoint}
B^*:Y^* &\into& (f_A(\sxx))^*\\
y^*&\mapsto& y^*B.\nonumber
\end{eqnarray}
is a well defined bounded linear operator and $\|B^*\|\leq Lip(B)$. The linear operator $B^*$ is called the adjoint of $B$.
\end{itemize}
\end{lemma}

The general procedure that we will apply to move from a given theory on linear operators $S:X\into Y$  to the broader context of multilinear operators $\Txxy$ is described as follows: First, understand a specific type of boundedness condition on linear operators $S$ as a continuous (or, equivalently, a Lipschitz) condition. Second, formulate such Lipschitz condition for the associated $\Sigma$-operators $f_T:\sxx\into Y$ and, finally, write this conditions in terms of multilinear mappings $T$ using the relation $f_T(x_1\otimes\cdots\otimes x_n)=T(x_1,\ldots,x_n)$.


\section{The Laprest\'e Tensor Norm}

In order to define the Laprest\'e norm of $u$ in $\xtxty$, especial representations of $u$ must be considered. These representations are those of the form $\sum_{i=1}^m\lambda_i(\pimqi)\tens y_i$ where $p_i$, $q_i$ and $y_i$, $1\leq i\leq m$, are in $\sxx$ and $Y$, respectively. Similar representations appeared for the first time in \cite{angulo10} and subsequently in \cite{fernandez-unzueta18b}.

\begin{definition}\label{laprestenorms}
Let $1\leq p, q\leq \infty$ such that $\frac{1}{p}+\frac{1}{q}\geq1$. Take the unique $r\in[1,\infty]$ determined by $1=\frac{1}{r}+\frac{1}{q^*}+\frac{1}{p^*}$. Let $\xxx$ and $Y$ be Banach spaces. We define the Laprest\'e norm $\apq^\pi$ on $\xtxty$ by
$$\apq^\pi(u):=\inf \left\{ \|(\lambda_i)_{i=1}^m\|_r \|(\pimqi)_{i=1}^m\|_{q^*}^{w,\pi} \|(y_i)_{i=1}^m\|_{p^*}^w  \Big|   u=\rlu \right\}.$$
\end{definition}

We include the proof that $\apq^\pi$ is actually a norm in Proposition \ref{tensornorm} where other properties are also presented. Throughout this section and the next we will assume this fact. Plainly, taking $n=1$ we have a generalization of the Laprest\'e norm for the case of two factors (see \cite[Sec. 12.5]{defant93}).

\begin{proposition}\label{link}
Let $\xxx$ and $Y$ be Banach spaces. The following are equivalent:
\begin{itemize}
	\item  [i)]   $\zeta$ is a bounded functional on $(\xtxty,\apq^\pi)$.
	\item  [ii)]  There exists $C>0$ such that
	$$\|(\lev \zeta ,  (\pimqi)\tens y_i   \rev   )_{i=1}^m\|_{r^*}\leq C\|(p_i-q_i)_{i=1}^m\|_{q^*}^{w,\pi}\|(y_i)_{i=1}^m\|_{p^*}^w$$
	for all finite sequences $(p_i)_{i=1}^m$, $(q_i)_{i=1}^m$ in $\sxxp$ and $(y_i)_{i=1}^m$ in $Y$.
\end{itemize}
In this case $\|\zeta\|=\inf C$ where the infimum is taken over all the constants $C$ as above.
\end{proposition}

\begin{proof}
(i)$\Rightarrow$(ii): Linearity of $\zeta$ and the fact $(\ell_r^m)^*=\ell_{r^*}^m$ imply
\begin{align*}
\|(\lev \zeta  ,  (\pimqi)\tens y_i \rev )_{i=1}^m\|_{r^*} &= \sup\limits_{\|(\lambda_i)_{i=1}^m\|_r\leq 1}  \left|\zeta\left(\sumim \lambda_i (\pimqi)\tens y_i\right)\right|\\
  &\leq \|\zeta\|\|(p_i-q_i)_{i=1}^m\|_{q^*}^{w,\pi}\|(y_i)_{i=1}^m\|_{p^*}^w.
\end{align*}
(ii)$\Rightarrow$(i): Let $u=\rlu$. Then
\begin{align*}
|\zeta(u)| =\left|\sumim \lambda_i \zeta((\pimqi)\tens y_i)\right|&\leq \|(\lambda_i)_{i=1}^m\|_r \|(\lev\zeta  ,     (\pimqi)\tens y_i \rev )_{i=1}^m\|_{r^*}\\
				 &\leq  C\|(\lambda_i)_{i=1}^m\|_r \|(p_i-q_i)_{i=1}^m\|_{q^*}^{w,\pi}\|(y_i)_{i=1}^m\|_{p^*}^w.
\end{align*}
\end{proof}

Next, we present a characterization of the bounded linear functionals on the normed space $(\xtxty,\apq^\pi)$. The linear analogous of this result is contained in the original proof of the characterization of $p$-dominated operators of S. Kwapie\'n, see \cite[Proposition 2]{kwapien72}.

\begin{theorem}\label{functionals}
Let $\xxx$, $Y$ be Banach spaces. The following are equivalent:
\begin{itemize}
	\item [i)]   $\zeta$ is a bounded linear functional on $\left(\xtxty,\apq^\pi\right)$.
	\item [ii)]  For any $w^*$-compact norming subsets $K\subset B_{\Lxx}$ and $L\subset B_{Y^*}$ there exist a nonnegative constant $C$ and probability regular Borel measures $\mu$ and $\nu$ on $K$ and $L$ respectively such that for all $a,b$ decomposable tensors in $\xtx$ and $y$ in $Y$
\begin{small}
$$\hspace{1.2cm}|\lev  \zeta  ,   (a-b)\tens y \rev| \leq   C\left(\int\limits_K |f_\psi(a)-f_\psi(b)|^{q^*}d\mu(\psi)\right)^{\frac{1}{q^*}}     \left(\int\limits_L |y^*(y)|^{p^*}d\nu(y^{*})\right)^{\frac{1}{p^*}}$$
\end{small}
(where the first integral is replaced by $\pi(a-b)$ if $q=1$ and the second by $\|y\|$ if $p=1$).
\end{itemize}
Under these circumstances $\|\zeta\|=\inf C$.
\end{theorem}

\begin{proof}
The case $p=1$ follows from \cite[Theorems 2.26, 3.1]{angulo10} (analogously the case $q=1$). We restrict our attention to $1<p,q$, hence, $r^*<\infty$.

Let $\zeta$ be a bounded linear functional such that $\|\zeta\|= 1$. Let $M_1^+(K)\subset C(K)^*$ and $M_1^+(L)\subset C(L)^*$ be the sets of probability measures on $K$ and $L$, respectively. Define $C:=M_1^+(K)\times M_1^+(L)\subset C(K)^*\times C(L)^*$. Notice that $C$ is a compact subset of $\left(C(K)^*,w^*\right)\times \left(C(L)^*,w^*\right)$.

For each $a,b$ in $\sxx$ and $y$ in $Y$ consider $I_{a,b}$ in $C(K)$ and $I_y$ in $C(L)$ defined by
\begin{eqnarray*}
I_{a,b}:(K,w^*)  &\into & \rr\\
                         \psi      &\mapsto &  |f_\psi(a)-f_\psi(b)|^{q^*},
\end{eqnarray*}
\begin{eqnarray*}
I_y:(L,w^*)  &\into & \rr\\
                         y^*      &\mapsto &  |y^*(y)|^{p^*}.
\end{eqnarray*}
Also consider $K_{C(K)}:C(K)\into C(K)^{**}$ and $K_{C(L)}:C(L)\into C(L)^{**}$. Then
$$ \lev K_{C(K)}(I_{a,b})  ,   \mu\rev=\mu(I_{a,b})=\int\limits_K |f_\psi(a)-f_\psi(b)|^{q^*}d\mu(\psi) \qquad\forall\, \mu\in C(K)$$
and
$$ \lev K_{C(L)}(I_y)  ,  \nu\rev=\nu(I_y)=\int\limits_L |y^*(y)|^{p^*}d\nu(y^*) \qquad\forall\, \nu\in C(L).$$
Hence, the function
\begin{eqnarray*}
H_{a,b,y}:C   &\into &  \rr\times \rr\\
                    (\mu,\nu)   &\mapsto &  \left(  \lev K_{C(K)}(I_{a,b})  ,   \mu\rev \,,\,  \lev K_{C(L)}(I_y)  ,  \nu\rev \right)
\end{eqnarray*}
is continuous. Define
$$f_{a,b}=\pi_1 \circ H_{a,b,y},$$
$$g_y=\pi_2\circ H_{a,b,y},$$
where $\pi_j:\rr\times \rr\into \rr$ is the $j$-th projection, $j=1,2$. Also, consider the constant function
\begin{eqnarray*}
c_{a,b,y}:C   &\into &  \rr\times \rr\\
                    (\mu,\nu)   &\mapsto &  |  \lev \zeta , (a-b)\tens y  \rev|^{r^*}.
\end{eqnarray*}
The functions $f_{a,b}$, $g_y$ and $c_{a,b,y}$ are continuous by construction. Moreover, it is a simple matter to prove that they are affine.

Let $\mathcal{F}$ be the set of all functions $f:C\into \rr$ for which there exist finite sequences $(p_i)_{i=1}^m$, $(q_i)_{i=1}^m$ in $\sxxp$ and $(y_i)_{i=1}^m$ in $Y$ such that
$$f=\sumim \frac{r^*}{q^*} f_{p_i,q_i}+ \frac{r^*}{p^*}g_{y_i}+c_{p_i,q_i,y_i}.$$
In particular, every $f$ in $\mathfrak{F}$ is upper semicontinuous and concave.

The set $\mathcal{F}$ is convex since if the functions $f_1=\sumim \frac{r^*}{q^*} f_{p_i,q_i}+ \frac{r^*}{p^*}g_{y_i}+c_{p_i,q_i,y_i}$ and $f_2=\sumim \frac{r^*}{q^*} f_{a_i,b_i}+ \frac{r^*}{p^*}g_{w_i}+c_{a_i,b_i,w_i}$ are in $\mathcal{F}$, then
\begin{align*}
\lambda_1f^1+\lambda_2f^2  =  \sumim \frac{r^*}{q^*} \left(f_{\lambda^\frac{1}{q^*}p_i,\lambda^\frac{1}{q^*}q_i}\right.&\left.+f_{\lambda^\frac{1}{q^*}a_i,\lambda^\frac{1}{q^*}b_i}\right)+ \frac{r^*}{p^*}\left(g_{\lambda^{\frac{1}{p\*}} y_i}+ g_{\lambda^\frac{1}{p^*}w_i}\right)\\
&+c_{\lambda^\frac{1}{q^*}p_i,\lambda^\frac{1}{q^*}q_i,\lambda^\frac{1}{p^*}y_i}+c_{\lambda^\frac{1}{q^*}a_i,\lambda^\frac{1}{q^*}b_i,\lambda^\frac{1}{p^*}w_i}
\end{align*}
holds for all $\lambda_1\geq 0$, $\lambda_2\geq 0$ such that $\lambda_1+\lambda_2=1$.

We claim that any $f$ in $\mathcal{F}$ is nonnegative in at least one point. To prove this, notice that every sequence $(y_i)_{i=1}^m$ defines a $w^*$-continuous function
\begin{eqnarray*}
Y^*	&\into &		  \rr\\
y^*	&\mapsto &  \left(\sumim |y^*(y_i)|^{p^*}\right)^{\frac{1}{p^*}}.
\end{eqnarray*}
Compactness of $L$ ensures the existence of $y_o^*$ such that
$$\|(y_i)_{i=1}^m\|_{p^*}^w= \left(\sumim |y_0^*(y_i)|^{p^*}\right)^{\frac{1}{p^*}}.$$
Analogously, there exists $\psi_0\in K$ such that
$$\|(\pimqi)_{i=1}^m\|_{q^*}^{w,\pi}=\left(\sumim |f_{\psi_0}(p_i)-f_{\psi_0}(q_i)|^{q^*}\right)^{\frac{1}{q^*}}.$$
For the Dirac measures $\delta_{\psi_0}$ on $K$ and $\delta_{y_0^*}$ in $L$ we have
\begin{align*}
f(\delta_{\psi_0},\delta_{y_0^*}) &=    \frac{r^*}{q^*}(\|(\pimqi)_{i=1}^m\|_{q^*}^{w,\pi})^{q^*} +\frac{r^*}{p^*}(\|(y_i)_{i=1}^m\|_{p^*}^w)^{q^*}-\sumim |\zeta( (\pimqi)\tens y_i)|^{r^*}\\
								                 &\geq	  (\|(\pimqi)_{i=1}^m\|_{q^*}^{w,\pi})^{\frac{q^*r^*}{q^*}}\,  (\|(y_i)_{i=1}^m\|_{p^*}^w)^{\frac{p^*r^*}{p^*}}-\sumim |\zeta( (\pimqi)\tens y_i)|^{r^*}\\
											          &=   \left(\|(\pimqi)_{i=1}^m\|_{q^*}^{w,\pi}\right)^{r^*}\,  (\|(y_i)_{i=1}^m\|_{p^*}^w )^{r^*}-\sumim |\zeta( (\pimqi)\tens y_i  ) |^{r^*}\\
								&\geq 0,
\end{align*}
where the first inequality follows from the fact $\frac{s}{c}+\frac{t}{c^*}\geq s^{\frac{1}{c}}t^{\frac{1}{c}}$ for all $s\geq 0$ and $t\geq 0$ and $1<c<\infty$ and the second from $\|\zeta\|= 1$.

Applying Ky Fan's lemma (see~\cite[A3]{defant93}) we obtain $(\mu,\nu)\in C$ such that
$$0\leq f(\mu,\nu)     \qquad\forall  f\in \mathcal{F}.$$
Hence
$$|\lev \zeta , (a-b)\tens y  \rev |^{r^*} \leq	\frac{r^*}{q^*}\int\limits_K |f_\psi(a)-f_\psi(b)|^{q^*}d\mu(\psi)+\frac{r^*}{p^*}\int\limits_L |y^*(y)|^{p^*}d\nu(y^*).$$

Notice that for any $s,t>0$ we have
\begin{align*}
| \lev\right. \zeta, (a-b)&\tens y  \left.\rev|=st\, | \lev \zeta  ,   s^{-1}(a-b)\tens t^{-1}y \rev | \\
                         &\leq st\left( \frac{r^*}{s^{q^*}q^*}\int\limits_K |f_\psi(a)-f_\psi(b)|^{q^*}d\mu(\psi)+\frac{r^*}{t^{p^*}p^*}\int\limits_L |y^*(y)|^{p^*}d\nu(y^*) \right)^{\frac{1}{r^*}}.
\end{align*}
Taking $s=\left(\int\limits_K |f_\psi(a)-f_\psi(b)|^{q^*}d\mu(\psi)\right)^{\frac{1}{q^*}}$ and $t=\left(\int\limits_L |y^*(y)|^{p^*}d\nu(y^*)\right)^{\frac{1}{p^*}}$ we obtain
$$| \lev\zeta  ,   (a-b)\tens y  \rev|\leq \left(\int\limits_K |f_\psi(b)-f_\psi(b)|^{q^*}d\mu(\psi)\right)^{\frac{1}{q^*}}\, \left(\int\limits_L |y^*(y)|^{p^*}d\nu(y^*)\right)^{\frac{1}{p^*}}$$
which is the inequality we were looking for. For the general case, an argument of normalization of $\zeta$ is enough. In this situation $\inf C\leq\|\zeta\|$.

Conversely, (ii) combined with H\"{o}lder inequality for $\frac{q^*}{r^*}$ and $\frac{p^*}{r^*}$ implies
\begin{small}
\begin{align*}
\|(\zeta((p_i&-q_i)\tens y_i))\|_{r^*} =\left(\sumim|\zeta((\pimqi)\tens y_i)|^{r^*}\right)^{\frac{1}{r^*}}\\
                                                                   &\leq C\,\left(\sumim\left(\int\limits_K |\psi(p_i)-\psi(q_i)|^{q^*}d\mu(\psi)\right)^{\frac{r^*}{q^*}}\left(\int\limits_L |y^*(y_i)|^{p^*}d\nu(y^*)\right)^{\frac{r^*}{p^*}}\right)^{\frac{1}{r^*}}\\
                                                                  &\leq C\,\left(\sumim\int\limits_K |\psi(p_i)-\psi(q_i)|^{q^*}d\mu(\psi)\right)^{\frac{1}{q^*}}\,\left(\sumim\int\limits_L |y^*(y_i)|^{p^*}d\nu(y^*)\right)^{\frac{1}{p^*}}\\
                                                                  &=C\,\left(\int\limits_K \sumim|\psi(p_i)-\psi(q_i)|^{q^*}d\mu(\psi)\right)^{\frac{1}{q^*}}\,\left(\int\limits_L \sumim |y^*(y_i)|^{p^*}d\nu(y^*)\right)^{\frac{1}{p^*}}\\
                                                                 &\leq   C\,\|(\pimqi)\|_{q^*}^{w,\pi} \|(y_i)\|_{p^*}^w.
\end{align*}
\end{small}
Proposition \ref{link} ensures that $\zeta$ is bounded and $\|\zeta\|\leq\inf C$.
\end{proof}


\section{$(p,q)$-Dominated Multilinear Operators}

As stated in Subsection 1.1, the procedure to translate the $(p,q)$-domination property for linear operators to the multilinear setting gives rise to the next definition.
\begin{definition}\label{dominatedoperators}
Let $1\leq p,q\leq \infty$ such that $\frac{1}{p}+\frac{1}{q}\leq 1$. Take the unique $r\in[1,\infty]$ such that $1=\frac{1}{r}+\frac{1}{p}+\frac{1}{q}$. The  multilinear operator $T:\xpx\into Y$ is called $\pi$-$(p,q)$-dominated if there exists a constant $C>0$ such that
\begin{align*}
 \|(\lev y_i^*  ,  T(x_i^1,\dots, x_i^n)\right.&\left.-T(z_i^1,\dots, z_i^n)  \rev)_{i=1}^m\|_{r^*}\\
	                                                &\leq C\|(x_i^1\tens\dots\tens x_i^n-z_i^1\tens\dots\tens z_i^n)_{i=1}^m\|_p^{w,\pi}  \|(y_i^*)_{i=1}^m\|_q^w
\end{align*}
holds for all finite sequences $(x_i^1,\dots, x_i^n)_{i=1}^m$, $(z_i^1,\dots, z_i^n)_{i=1}^m$ in $\xpx$ and $(y_i^*)_{i=1}^m$ in $Y^*$. Define $D_{p,q}^\pi(T)$ as the infimum of all the constants $C$ as above.
\end{definition}

Proposition \ref{link} allows us to translate the $\pi$-$(p,q)$-dominated property of multilinear operators to the tensorial context (see also Theorem \ref{tensorialrepresentation}).

\begin{proposition}\label{isometry}
The multilinear operator $\Txxy$ is $\pi$-$(p,q)$-dominated if and only if the functional
\begin{eqnarray*}
\zeta_T:(\xtxtyd, \alpha_{q^*,p^*}^\pi) &\into & \kk\\
\xxt\tens y^* &\mapsto &  \lev  y^*  ,  T\xxp  \rev
\end{eqnarray*}
is bounded. In this case $D_{p,q}^\pi(T)=\|\zeta_T\|$.
\end{proposition}

\begin{proof}
First, notice that $\frac{1}{p}+\frac{1}{q}\leq 1$ implies $\frac{1}{p^*}+\frac{1}{q^*}\geq 1$. Hence, $\alpha_{q^*,p^*}^\pi$ and $D_{p,q}^\pi$ make sense. Moreover, $1=\frac{1}{r}+\frac{1}{p}+\frac{1}{q}$ is valid for $D_{p,q}^\pi$ and $\alpha_{q^*,p^*}^\pi$. Let $(p_i)_{i=1}^m$, $(q_i)_{i=1}^m$ and $(y_i^*)_{i=1}^m$. The proof is complete by noticing that
$$\|( \lev  \zeta_T,   (\pimqi)\tens y_i^*\rev  )_{i=1}^m\|_{r^*} =   \|( \lev   y_i^*  ,  f_T(p_i)-f_T(q_i)  \rev  )_{i=1}^m\|_{r^*}$$
and applying Proposition \ref{link}.
\end{proof}

The implications to multilinear operators of the results obtained in the previous section are reflected in the next theorem.

\begin{theorem}\label{factorization}(Kwapie\'n's Factorization Theorem)
Let $\xxx, Y$  be Banach spaces. The following are equivalent:
\begin{itemize}
	\item [i)]   $T:\xpx\into Y$ is $\pi$-$(p,q)$-dominated.
	\item [ii)]  For any $w^*$-compact norming subsets $K\subset B_{\Lxx}$ and $L\subset B_{Y^{**}}$ there exists a positive constant $C$ such that
	\begin{align}\label{factorization2}
\hspace{1cm}|\lev  y^*   \right.&, \left.  T(x)-T(z) \rev|\nonumber\\
				  &\leq   C\left(\int\limits_K |\psi(x)-\psi(z)|^{p}d\mu(\psi)\right)^{\frac{1}{p}}     \left(\int\limits_L |y^{**}(y^*)|^{q}d\nu(y^{**})\right)^{\frac{1}{q}}
\end{align}
for all $x,z$ in $\xpx$ and $y^*$ in $Y^*$ (the first integral is replaced by $\pi(\tens(x)-\tens (z))$ if $p=\infty$ and the second by $\|y^*\|$ if $q=\infty$).
	\item [ii)]  $T$ factors as follows
\begin{small}
\begin{eqnarray}\label{factorization1}
\begin{array}{c}
\xymatrix{
\xpx\ar[dr]_{A}\ar[rr]^-T &                         & Y \\
                                            &f_A(\sxxp)\ar[ur]_B\ar[d]  &     \\
                                            &         Z              &
}
\end{array},
\end{eqnarray}
\end{small}
where $Z$ is a Banach space, $A:\xxx\into Z$ is a Lipschitz $p$-summing multilinear operator and $B:f_A(\sxxp)\into Y$ is a Lipschitz function whose adjoint linear operator $B^*:Y^*\into (f_A(\sxxp))^*$ is $q$-summing.
\end{itemize}
Under these circumstances $D_{p,q}^\pi(T)=\inf C=\inf \pi_p^{Lip}(A)\,  \pi_q(B^*)$ where the infimums are taken over all $C$ as in \eqref{factorization2} and all possible factorizations as in (\ref{factorization1}), respectively.
\end{theorem}

\begin{proof}
(i)$\Rightarrow$(ii): Proposition \ref{isometry} asserts that $\zeta_T$ is a bounded functional on the normed space  $\left(\xtxtyd,\alpha_{q^*,p^*}^\pi\right)$. By Theorem~\ref{functionals} there exist measures $\mu$, $\nu$ on $K$ and $L$ respectively such that
\begin{align*}
|\lev\right.  y^*,  f_T&(a)-f_T(b) \left.\rev| = |\lev \zeta ,  (a-b)\tens y^* \rev  | \\
				  &\leq   D_{p,q}^\pi(T)\,\left(\int\limits_K |f_\psi(a)-f_\psi(b)|^{p}d\mu(\psi)\right)^{\frac{1}{p}}     \,\left(\int\limits_L |y^{**}(y^*)|^{q}d\nu(y^{**})\right)^{\frac{1}{q}}.
\end{align*}

(ii)$\Rightarrow$(iii): Define
\begin{eqnarray*}
A:\xpx        &\into &	    L_p(\mu)\\
	\xxp	  &\mapsto &  j_{p}\iota\xxp,
\end{eqnarray*}
where $\iota:\xpx\into C(K)$ acts by evaluation on $K$ and $j_{p}:C(K)\into L_{p}(\mu)$ is the canonical map. Plainly, $A$ is multilinear and bounded since $j_p$ is linear and bounded and $\iota$ is multilinear and bounded. Even more, $A$ is Lipschitz $p$-summing  and $\pi_{p}^{Lip}(A)\leq \|\iota\|\pi_{p}(j_{p})=1$ since $j_p$ is $p$-summing and $\mu$ is a probability measure (see \cite[Proposition 2.5]{angulo18}).

On the other hand, define
\begin{eqnarray*}
B:f_A(\sxxp) 	 &\into &	       Y\\
     f_A(a)      &\mapsto &  f_T(a).
\end{eqnarray*}
Inequality \eqref{factorization2} ensures that $B$ is well defined and
$$|  \lev B^*y^*  ,   f_A(a)-f_A(b) \rev|\leq C  \|f_A(a)-f_A(b)\|_{L_p(\mu)}  \left(\int\limits_L |y^{**}(y^*)|^{q}d\nu(y^{**})\right)^{\frac{1}{q}}.$$
Hence, $B^*(y^*)$ is a Lipschitz function and
$$Lip(B^*(y^*))\leq C\,  \left(\int\limits_L |y^{**}(y^*)|^{q}d\nu(y^{**})\right)^{\frac{1}{q}}.$$
The Pietsch's Domination Theorem asserts that $B^*:Y^*\into (f_A(\sxxp))^*$ is a linear $q$-summing operator with $\pi_{q}(B^*)\leq C$. This way, $T=BA$ and $\pi_p^{Lip}(A)  \pi_q(B^*)\leq C$.

(iii)$\Rightarrow$(i): Let $A$ and $B$ as in \eqref{factorization1}. H\"{o}lder's inequality applied to he conjugate indexes $\frac{p}{r^*}$ and $\frac{q}{r^*}$ implies
\begin{small}
\begin{align*}
\| ( \lev  y_i^*,   f_T(p_i)-f_T(q_i) \rev  )_{i=1}^m\|_{r^*}&= \|\left(B^*y_i^* f_A(p_i)-  B^*y_i^*f_A(q_i)  \right)_{i=1}^m\|_{r^*}\\
								&=\left(\sumim| B^*y_i^* (f_A(p_i))-  B^*y_i^*(f_A(q_i)) |^{r^*} \right)^{\frac{1}{r^*}}\\
								&\leq\left(\sumim Lip(B^*y_i^*)^{r^*}\,   \|f_A(p_i)-f_A(q_i)\|^{r^*} \right)^{\frac{1}{r^*}}\\
								&\leq\left(\sumim Lip(B^*y_i^*)^{q}\right)^{\frac{1}{q}}\left(\sumim \|f_A(p_i)-f_A(q_i)\|^{p} \right)^{\frac{1}{p}}\\
								&\leq 	\pi_{q}(B^*)  \pi_{p}^{Lip}(A)  \|(\pimqi)\|_{p}^{w,\pi}  \|(y_i^*)\|_{q}^w.
\end{align*}
\end{small}
Hence $T$ is $\pi$-$(p,q)$-dominated and $D_{p,q}^\pi(T) \leq \pi_{p}^{Lip}(A)\pi_{q}(B^*) $.
\end{proof}

As well as in the case of dominated linear operators, the cases $q=\infty,p^*$ are of particular interest. The case $q=p^*$ is a generalization of the $p$-dominated linear operators originally studied by Kwapie\'n in  \cite{kwapien72}. A good exposition of the results of Kwapie\'n can be found in the monograph \cite[Chapter 9]{diestel95}. The case $q=\infty$ is detailed next.

\begin{corollary}\label{psumming}
Let $\Txxy$ be a bounded multilinear operator. Then, $T$ is $\pi$-$(p,\infty)$-dominated if and only if $T$ is Lipschitz $p$-summing. In this case, $\pi_p^{Lip}(T)=D_{p,\infty}^\pi(T)$.
\end{corollary}

\begin{proof}
First recall that for linear operators the properties $\infty$-summing and continuity are the same and $\pi_\infty(\cdot)=\|\cdot\|$. The Lipschitz $p$-summing property is preserved by compositions as in \eqref{factorization1}. Hence, by (iii) of previous theorem, $\pi$-$(p,\infty)$-domination implies Lipschitz $p$-summability and $\pi_p^{Lip}(T)= \pi_p^{Lip}(BA)\leq \pi_p^{Lip}(A)\|B^*\|$ for any factorization of $T$ as in \eqref{factorization1}. Then, $\pi_p(T)\leq D_{p,\infty}^\pi(T)$. Conversely, if $T$ is Lipschitz $p$-summing, then $T=IT$ is $\pi$-$(p,\infty)$-dominated, where $I$ is the inclusion of $f_T(\sxxp)$ into $Y$. Moreover, $D_{p,\infty}^\pi(T)\leq\pi_p(T)$.
\end{proof}

Let $\mathcal{D}_{p,q}^\pi(\xxx;Y)$ be the set of all $\pi$-$(p,q)$-dominated multilinear operators from $\xpx$ to $Y$. Proposition \ref{isometry} says that $\mathcal{D}_{p,q}^\pi(\xxx;Y)$ is a vector space if we endow it with the sum and multiplication by scalars defined pointwise. Even more, $D_{p,q}^\pi$ is a norm on it.

\begin{proposition}\label{complete}
$\mathcal{D}_{p,q}^\pi(\xxx;Y)$ is a Banach space.
\end{proposition}

\begin{proof}
First, notice that if $T$ is $\pi$-$(p,q)$-dominated then
$$|\lev  y^*  ,  T\xxp\rev|\leq D_{p,q}^\pi(T)\|x^1\|\dots\|x^n\|\|y^*\|$$
holds for all $\xxp$ in $\xpx$ and $y^*$ in $Y^*$. Hence $\|T\|\leq D_{p,q}^\pi(T)$.

Let $(T_k)_{k}$ be a Cauchy sequence in $\mathcal{D}_{p,q}^\pi(\xxx;Y)$. By the comments above there exist a bounded multilinear operator $\Txxy$ such that $(T_k)_k$ converges pointwise to $T$. By Proposition \ref{isometry}, $(\zeta_{T_k})_k$ is a Cauchy sequence in $(\xtxtyd,\apq^\pi)^*$. Then, there exists $\zeta$ in $(\xtxtyd,\apq^\pi)^*$ such that $(\zeta_{T_k})_k$ converges to $\zeta$. We finish the proof by showing that $\zeta=\zeta_T$. This is clear since
$$\zeta (\xxt\tens y^*) =\lim\limits_{k\rightarrow\infty}\zeta_{T_k}(\xxt\tens y^*)=\lev  y^*  ,  T\xxp\rev$$
holds for all $y^*$ in $Y^*$ and $\xxp$ in $\xpx$.
\end{proof}


\section{Maximal Ideal Behavior of the Class $\mathcal{D}_{p,q}$ }

Consider Banach spaces $\xxx$ and subspaces $E_i$ of $X_i$ for $1\leq i \leq n$. Recall that, in general, $E_1\hat{\tens}_\pi\dots\hat{\tens}_\pi E_n$ is not a subspace of $\xtxp$. As a consequence, the $\pi$-$(p,q)$-dominated norm of $T:\xpx\into Y$ might not be compatible with the $\pi$-$(p,q)$-dominated norm of the restriction of $T$ to $\epe$. For a well behavior we extend our definition of dominated operators regarding any reasonable crossnorm $\beta$ on $\xtx$.

A norm $\beta$ on $\xtx$ is named a reasonable crossnorm if
$$\varepsilon(u)\leq\beta(u)\leq\pi(u)\qquad\forall \, u\in \xtx,$$
where $\pi$ and $\varepsilon$ denote the projective and injective tensor norms respectively. Here, $\varepsilon(u)=\sup\{ |x_1^*\tens\dots\tens x_n^* (u)| \,|\,  x_i^*\in B_{X_i^*}, 1\leq i \leq n \}$ for all $u$ in $\xtx$. For a theory of tensor norms of this type, the reader may check \cite{floret_hunfeld01}. In the sequel, $\beta$ denotes a reasonable crossnorm on $\xtx$ and $\xtxb$ is the resulting normed space. Let $\Lbxx$ denote the Banach space of all multilinear forms $\fhi:\xpx\into \kk$, whose linearization $\tilde{\fhi}:\xtxb\into \kk$ is bounded, endowed with the norm $\|\fhi\|_\beta:=\|\widetilde{\fhi}\|$. Let $\|\cdot\|_{p}^{w,\beta}$ denote the weak $p$-norm of weak $p$-summable sequences in $\xtxb$.

\subsection{Local Behavior of $(p,q)$-Dominated Multilinear Operators}

\begin{definition}\label{betalaprestenorms}
Let $1\leq p,q\leq \infty$ such that $\frac{1}{p}+\frac{1}{q}\geq1$. Take the unique $r\in[1,\infty]$ with the property $1=\frac{1}{r}+\frac{1}{q^*}+\frac{1}{p^*}$. Let $\xxx$ and $Y$ be Banach spaces and $\beta$ be a reasonable crossnorm on $\xtx$. We define the Laprest\'e norm $\apqb$ on $\xtxty$ by
$$\apqb(u):=\inf \left\{ \|(\lambda_i)_{i=1}^m\|_r \|(\pimqi)_{i=1}^m\|_{q^*}^{w,\beta} \|(y_i)_{i=1}^m\|_{p^*}^w  \Big|  u=\rlu\right\}.$$
\end{definition}

\begin{definition}\label{betadominatedoperators}
Let $1\leq p,q \leq \infty$ such that $\frac{1}{p}+\frac{1}{q}\leq 1$. Take the unique $r\in[1,\infty]$ such that $1=\frac{1}{r}+\frac{1}{p}+\frac{1}{q}$. The  multilinear operator $T:\xpx\into Y$ is called $\beta$-$(p,q)$-dominated if there exists a constant $C>0$ such that
\begin{align*}
  \|(\lev y_i^*  ,  T(x_i^1,\dots, x_i^n)\right.-&\left. T(z_i^1,\dots, z_i^n)  \rev)_{i=1}^m\|_{r^*}\\
	                                                &\leq C\|(x_i^1\tens\dots\tens x_i^n-z_i^1\tens\dots\tens z_i^n)_{i=1}^m\|_{p}^{w,\beta}  \|(y_i^*)_{i=1}^m\|_{q}^w
\end{align*}
holds for all finite sequences $(x_i^1,\dots, x_i^n)_{i=1}^m$, $(z_i^1,\dots, z_i^n)_{i=1}^m$ in $\xpx$ and $(y_i^*)_{i=1}^m$ in $Y^*$. Define $D_{p,q}^\beta(T)$ as the infimum of the constants $C$ as above.
\end{definition}

All results of $\pi$-$(p,q)$-dominated multilinear operators and the Laprest\'e tensor norm $\apq^\pi$ are valid if we replace $\pi$ by $\beta$. This is, the $\beta$-$(p,q)$-dominated operators are in duality with the Laprest\'e tensor norms $\apq^\beta$ (Proposition \ref{isometry}); The $\beta$-$(p,q)$-dominated operators are characterize in terms of dominations (where $\Lxx$ is replaced by $\Lbxx$) and factorizations (where Lipschitz $p$-summing operators are replaced by $\beta$-Lipschitz $p$-summing multilinear operators, see Theorem \ref{factorization}). The proofs of these results work exactly as in previous sections. For avoiding repetitions we do not explicit announce them; although, we make use of them in this more general version.

It is not difficult to see that if  $E_i$ is a closed subspace of $X_i$, $1\leq i\leq n$, then the restriction of $\beta$ to $\ete$ is a reasonable crossnorm. This restriction is denoted by $\br$. An application of the Hahn-Banach theorem shows that for every pair of sequences $(p_i)_{i=1}^m$ and $(q_i)_{i=1}^m$ in $\Sigma_{\eee}$ we have
\begin{equation*}
\sup\limits_{\fhi\in B_{\Lbxx}}\left(\sumim |\fhi(p_i)-\fhi(q_i)|^p\right)^{\frac{1}{p}} = \sup\limits_{\phi\in B_{\mathcal{L}^{\br}(\eee)}} \left(\sumim |\phi(p_i)-\phi(q_i)|^p\right)^{\frac{1}{p}},
\end{equation*}
in other words
\begin{equation}\label{key}
\|(\pimqi)_{i=1}^m;\xtx\|_p^{w,\beta}= \|(\pimqi)_{i=1}^m;\ete\|_p^{w,\br}.
\end{equation}

\begin{proposition}\label{finitelygenerated}
Let $\xxx$, $Y$ be Banach spaces and $\beta$ be a reasonable crossnorm on $\xtx$. Then
$$\apqb(u; \xxx,Y)=\inf \apq^{\br}(u; \eee,F),$$
where the infimum is taken over all finite dimensional subspaces $E_i$ and $F$ of $X_i$, $1\leq i\leq n$ and $Y$, respectively such that $\etetf$ contains $u$.
\end{proposition}

\begin{proof}
Let $E_i$ and $F$ as above. Let $u$ in $\etetf$ and $\sum_{i=1}^m \lambda_i (\pimqi)\tens y_i$ be a representation of $u$ in $\etetf$. Hence, \eqref{key} implies
$$\apqb(u;\xxx,Y)\leq\inf\apq^\br(u;\eee,F).$$

On the other hand, for any $u$ in $\xtxty$ and $\eta>0$ there exists a representation $\sum_{i=1}^m \lambda_i (\pimqi)\tens y_i$ of $u$ such that
$$\|(\lambda_i)_{i=1}^m\|_r\|(\pimqi)_{i=1}^m\|_{q^*}^{w,\beta}\|(y_i)_{i=1}^m\|_{p^*}^w \leq \apqb(u;\xtxty)+\eta.$$

It is clear that there exist finite dimensional subspaces $E_i$ and $F$ of $X_i$ and $Y$ respectively such that $p_i,q_i\in\Sigma_{\eee}^{\br}$ and $y_i\in F$ for $1\leq i\leq m$. Hence \eqref{key} and $\|(y_i)_{i=1}^m;F\|_{p^*}^w=\|(y_i)_{i=1}^m\|_{p^*}^w$ assert
$$\apq^\br(u;\eee, F)\leq \apqb(u;\xxx, Y)+\eta.$$
Therefore $\inf\apq^\br(u;  \eee, F)\leq \apqb(u;\xxx, Y)$.
\end{proof}

Let $E_i$ be a closed subspace of $X_i$ for $1\leq i\leq n$. Let
\begin{eqnarray*}
I_{\eee}:\epe  &\into       & \xtxb\\
                \xxp &\mapsto & \xxt.
\end{eqnarray*}
If $L$ is a closed subspace of $Y$, $Q_L:Y\into Y/L$ denote the canonical quotient linear map. Notice that the composition
$$Q_L f_T I_{\eee}:\epe\into Y/L$$
makes sense since the image of $I_{\eee}$ is contained in $\sxx$.
\begin{theorem}\label{maximal}
Let $\xxx$ be Banach spaces and $\beta$ be a reasonable crossnorm on $\xtx$. The following conditions are equivalent:
\begin{itemize}
\item[i)] $T:\xpx\into Y$ is $\beta$-$(p,q)$-dominated.
\item[ii)] There exist a constant $C>0$ such that for all finite dimensional subspace $E_i$ of $X_i$, $1\leq i\leq n$ and all finite codimensional subspace $L$ of $Y$ it is verified $D_{p,q}^\br(Q_L T I_{\eee})\leq C$.
\end{itemize}
In this situation $D_{p,q}^\beta(T)=\sup D_{p,q}^\br(Q_L f_T I_{\eee})$ where the suprema is taken over all $E_i$ and $L$ as in (ii).
\end{theorem}

\begin{proof}
Suppose that $T$ is $\beta$-$(p,q)$-dominated. Let $E_i$ and $L$ as in (ii). Let $(p_i)_{i=1}^m$ and $(q_i)_{i=1}^m$ in $\Sigma_{\eee}^\br$ and $(z_i^*)_{i=1}^m$ in $(Y/L)^*$. Hence, by \eqref{key}
\begin{align*}
\|(\lev  z_i^*  , Q_Lf_T I_{\eee}  (p_i) - Q_Lf_T\right.&\left. I_{\eee} \rev (q_i))_{i=1}^m\|_{r^*}\\
& = \|(\lev  Q_L^* z_i^*  , f_T (p_i)-f_T(q_i) \rev)_{i=1}^m\|_{r^*}\\
&\leq D_{p,q}^\beta(T) \|(\pimqi)_{i=1}^m\|_p^{w,\beta} \|(Q_L^* z_i^*)_{i=1}^m\|_q^w\\
&\leq D_{p,q}^\beta(T) \|(\pimqi)_{i=1}^m\|_p^{w,\br} \|( z_i^*)_{i=1}^m\|_q^w.
\end{align*}
Then the composition $Q_Lf_TI_{\eee}$ is a $\br$-$(p,q)$-dominated multilinear operator and $D_{p,q}^\br(Q_Lf_TI_{\eee})\leq D_{p,q}^\beta(T)$.

For the converse we apply Proposition \ref{isometry} in its version of the norm $\beta$. Let $u=\sum_{i=1}^m \lambda_i (\pimqi)\tens y_i^*$ and $\eta>0$. By the finitely generated property of $\alpha_{q^*,p^*}^\beta$ (see Proposition \ref{finitelygenerated}) there exist finite dimensional subspaces $E_i$ and $F$ of $X_i$, $1\leq i \leq n$ and $Y^*$ respectively such that $u\in \etetf$ and
$$\alpha_{q^*,p^*}^{\br}(u;\eee, F)\leq (1+\eta)\alpha_{q^*,p^*}^\beta(u;\xxx,Y^*).$$

Consider the finite codimensional subspace $L$ of $Y$ such that $(Y/L)^*$ is isometric to $F$ via $Q_L^*$. Let $\psi_i$ in $(Y/L)^*$ such that $(Q_L)^*(\psi_i)=y_i$, $1\leq i \leq m$. Proposition \ref{isometry} for the case $\br$ asserts that $\zeta_{Q_Lf_T I_{\eee}}:(\ete\tens (Y/L)^*,\alpha_{q^*,p^*})\into \kk$ is bounded and $\|\zeta_{Q_Lf_T I_{\eee}}\|=D_{p,q}^\br(Q_Lf_T I_{\eee})$.

Algebraic manipulations lead to $\zeta_T(u)=\zeta_{Q_Lf_TI_{\eee}} (\sum_{i=1}^m \lambda_i    (  \pimqi )\tens \psi_i )$. Then
\begin{align*}
|\zeta_T(u)|&\leq \|\zeta_{Q_Lf_TI_{\eee}}\|\alpha_{q^*,p^*}^\br\left(\sumim \lambda_i    (  \pimqi )\tens \psi_i;\eee,(Y/L)^*\right)\\
               &\leq   C \alpha_{q^*,p^*}^\br( u;\eee,F)\\
               &\leq   C (1+\eta)\alpha_{q^*,p^*}^\beta(u;\xxx,Y).
\end{align*}
This way $\|\zeta_T\|\leq C(1+\eta)$. That is, $T$ is $\beta$-$(p,q)$-dominated and $D_{p,q}^\beta(T)\leq C$.
\end{proof}

\subsection{Representation of the Class $\mathcal{D}_{p,q}$ in Terms of the Laprest\'e Tensor Norm $\mathbf{\apq}$}

In this subsection we exhibit the maximal ideal nature of the collection of all $(p,q)$-dominated multilinear operators. In particular, we prove that the property of being $(p,q)$-dominated is preserved under certain compositions. In duality, the Laprest\'e tensor norms verify a particular uniform property. To prove these results we need one more result about $\Sigma$-operators.

Let $\sxxb$ denote the metric space $\sxx$ endowed with the metric induced by $\beta$. From \cite[Theorem 2.1]{fernandez-unzueta18a}, if $\Txxy$ is a bounded multilinear operator, then
\begin{eqnarray}\label{betalip}
f_T:\sxxb &\into & Y\nonumber\\
         a   &\mapsto &  \tlin(a)
\end{eqnarray}
is Lipschitz and $\|T\|=Lip^\pi(f_T)\leq \Lb(f_T)\leq 2^{n-1} Lip^\pi (f_T)$.

Let $n$ in $\mathbb{N}$ and $1 \leq p, q\leq \infty$ such that $\frac{1}{p}+\frac{1}{q}\leq 1$. Also let $\xxx$ and $Y$ be Banach spaces and $\beta$ be a reasonable crossnorm on $\xtx$. Let $\mathcal{D}_{p,q}^\beta(\xxx; Y)$ denote the set of $\beta$-$(p,q)$-dominated multilinear operators form $\xpx$ to $Y$. Proposition \ref{complete} (for the case $\beta$) implies that $\mathcal{D}_{p,q}^\beta(\xxx; Y)$ is a Banach space endowed with the norm $D_{p,q}^\beta$.

The class $\mathcal{D}_{p,q}$ is defined as the union over all possible Banach spaces of the form $\mathcal{D}_{p,q}^\beta(\xxx; Y)$. That is
$$\mathcal{D}_{p,q}=\bigcup \mathcal{D}_{p,q}^\beta(\xxx;Y).$$
Also define $D_{p,q}$ as the application on $\mathcal{D}_{p,q}$ whose restriction to $\mathcal{D}_{p,q}^\beta(\xxx;Y)$ is the norm $D_{p,q}^\beta$.

\begin{proposition}\label{ideal}
Let $\xxx$, $Y$ be Banach spaces and $\beta$ be a reasonable crossnorm on $\xtx$. Then:
\begin{itemize}
    \item [i)]  $D_{p,q}^\beta(\fhi\cdot y)\leq\|\fhi\|_\beta\;\|y\|$ for all $\fhi\in\Lbxx$ and $y\in Y$.\\
    \item [ii)]  $\Lb(f_T) \leq D_{p,q}^\beta(T)$ for all $T\in \Dpqbxxy$.\\
    \item [iii)]  Let $\zzz, W$ be Banach spaces and $\theta$ be a reasonable cross norm on $\ztz$. If in the composition
    $$\begin{array}{c}
\xymatrix{
\szzt\ar[r]^-{f_R}   &   \sxxb\ar[r]^-{f_T}   &   Y\ar[r]^-S   &   W \\
}
\end{array}$$
     $R:\zpz\into\xtxb$ is a bounded multilinear operator , $T$ is an element of $\Dpqbxxy$ and $S:Y\into W$ is a bounded linear operator such that $\tilde{R}:\ztzt\into \xtxb$ is bounded and $f_R(\szz^\theta)\subset\sxxb$, then $Sf_T R$ belongs to $\mathcal{D}_{p,q}^\theta(\zzz; W)$ and $D_{pq}^\theta(Sf_T R)\leq \|\rlin\|D_{p,q}^\beta(T)\|S\|$.
\end{itemize}
\end{proposition}

\begin{proof}
(i) and (ii) follows immediately from definition. We only prove (iii). Let $(p_i)_{i=1}^m$ and $(q_i)_{i=1}^m$ be finite sequences in $\Sigma_{\zzz}^\theta$ and $(w_i^*)_{i=1}^m$ in $W^*$. Then
\begin{align*}
\|(\lev w_i^* , Sf_T f_R (p_i)-Sf_Tf_R(q_i) \rev&)_{i=1}^m\|_{r^*} = \|(\lev S^*w_i^* , f_T f_R (p_i)-Sf_Tf_R(q_i) \rev)_{i=1}^m\|_{r^*}\\
  &\leq D_{p,q}^\beta(T) \|(f_R(p_i)-f_R(q_i))_{i=1}^m\|_p^{w,\beta}\|(S^*w_i^*)_{i=1}^m\|_q^w \\
  &\leq \|\tilde{R}\|D_{p,q}^\beta(T)\|S\|\|(p_i-q_i)_{i=1}^m\|_p^{w,\theta}\|(w_i^*)_{i=1}^m\|_q^w
\end{align*}
asserts that $Sf_T R$ is $\theta$-$(p,q)$-dominated and $D_{p,q}^\theta(Sf_T R)\leq \|\tilde{R}\|D_{p,q}^\beta(T)\|S\|$.
\end{proof}

Bearing in mind Proposition \ref{ideal} and Theorem \ref{maximal} we may say that the pair $[\mathcal{D}_{p,q},D_{p,q}]$ has a maximal ideal behavior.

Let $1 \leq p, q\leq \infty$ such that $\frac{1}{p}+\frac{1}{q}\geq 1$. Let $\apq$ be an assignment defined as next: to each arrangement $(n,\xxx, Y, \beta)$, $\apq$ assigns the norm $\apqb$ on $\xtxty$, where $n$ is a natural number, $X_i$, $1 \leq i\leq n$, $Y$ are Banach spaces and $\beta$ is a reasonable crossnorm on $\xtx$. The assignment $\apq$ is called the Laprest\'e tensor norm.

\begin{proposition}\label{tensornorm}
Let $\xxx, Y$ be Banach spaces and $\beta$ be a reasonable crossnorm on $\xtx$. Then:
\begin{itemize}
	\item [i)] $\apqb$ is a norm on $\xtxty$.	
	\item [ii)]   $\apqb((a-b)\tens y)\leq\beta(a-b)\,  \|y\|$ for all $a,b\in\sxxb$ and $y\in Y$.
	\item [iii)]  Let $\fhi\in\Lbxx$ and $y^*\in Y^*$. The functional generated by
	\begin{eqnarray*}
\fhi\tens y^*:\left(\xtxty,\apqb\right) &\into &		  \kk\\
\xxt\tens y &\mapsto &  f_\fhi(\xxt)\,  y^*(y)
\end{eqnarray*}
is bounded and $\Lb(f_\fhi)\|y^*\|\leq\|\fhi\tens y^*\|\leq \|\fhi\|_\beta\,  \|y^*\|$.
	\item [iv)] Let $\zzz, W$ be Banach spaces and $\theta$ be a reasonable cross norm on $\ztz$. If $R:\zpz\into\xtxb$ is a bounded multilinear operator such that $\tilde{R}:\ztzt\into \xtxb$ is bounded and $f_R(\szz^\theta)\subset\sxx^\beta$ and if $S:W\into Y$ is a bounded linear operator, then
	\begin{eqnarray*}
R\tens S:\left(\ztztw, \apq^\theta\right) &\into &		  \left(\xtxty,\apqb\right)\\
\zzt\tens w &\mapsto &  f_R(\zzt)\tens S(w)
\end{eqnarray*}
is bounded and $\|R\tens S\|\leq \|\rlin\|\,  \|S\|$.
\end{itemize}
\end{proposition}

\begin{proof}
i): The triangle inequality and homogeneous property follows from standard arguments. Lets prove that $\apqb(u)=0$ implies $u=0$. For this end let $u=\sum_{i=1}^m \lambda_i (\pimqi)\tens y_i$, $\fhi\in\Lbxx$ and $y^*\in Y^*$. The H\"{o}lder's inequality implies
\begin{eqnarray}\label{laprestenorms1}
|\lev\fhi\tens y^*, u\rev| &\leq& \left(\sumim |\lambda_i|^r \right)^{\frac{1}{r}} \left(\sumim |f_\fhi(p_i)-f_\fhi(q_i)|^{q^*}\right)^{\frac{1}{q^*}}\left(\sumim |y^*(y_i)|^{p^*}\right)^{\frac{1}{p^*}} \nonumber\\
                          &\leq& \|\fhi\|_\beta\,  \|y^*\|\,  \|(\lambda_i)\|_r\; \|(\pimqi)\|_{q^*}^{w,\beta}\;  \|(y_i)\|_{p^*}^w.
\end{eqnarray}
Therefore, since $\beta$ is a reasonable crossnorm, the condition $\apqb(u)=0$ implies that $x_1^*\tens\dots\tens x_n^*\tens y^*(u)=0$ for all $x_i^*\in X_i^*$, $1\leq i \leq n$ and $y^*\in Y^*$. Hence, $u$ must be zero.

(ii): Immediately form definition.

(iii): Actually (\ref{laprestenorms1}) proves that $\fhi\tens y^*$ is bounded and $\|\fhi\tens y^*\|\leq \|\fhi\|_\beta\|y^*\|$. For the another inequality let $\eta>0$ and take $a$ and $b$ in $\sxxb$ such that $\Lb(f_\fhi)(1-\eta)\leq |f_\fhi(a)-f_\fhi(b)|$ and $\beta(a-b)\leq 1$. Analogously, choose $y$ in $B_{Y}$ with $\|y^*\|(1-\eta)\leq |y^*(y)|$. Then
$$\Lb(f_\fhi)\,  \|y^*\|\,  (1-\eta)^2  \leq \|\fhi\tens y^*\|\,  \apqb((a-b)\tens y)\leq \|\fhi\tens y^*\|.$$

(iv): First, notice that
$$\|(f_R(p_i)-f_R(q_i))\|_{q^*}^{w, \beta}\leq \|\rlin\|\,  \|(\pimqi))\|_{q^*}^{w, \theta}$$
holds for all finite sequences $(p_i)$ and $(q_i)$ in $\szz^\theta$. Therefore, for any representation $\sum_{i=1}^m \lambda_i (\pimqi)\tens y_i$ of $v$ in $\ztztw$ we have
\begin{align*}
\apqb\left(R\tens S (v);\xxx, Y\right)  &=     \apqb\left( \sumim \lambda_i(f_R(p_i)-f_R(q_i))\tens S(y_i)\right) \\
                     &\leq  \|(\lambda_i)\|_r\; \|(f_R(p_i)-f_R(q_i))\|_{q^*}^{w, \beta}\; \|(S(y_i))\|_{p^*}^w\\
                     &\leq  \|\rlin\|\; \|S\|\; \|(\lambda_i)\|_r\; \|(\pimqi))\|_{q^*}^{w, \theta}\; \|(y_i)\|_{p^*}^w.
\end{align*}
This asserts that $R\tens S: (\ztztw,\alpha_{p,q}^\theta)\into(\xtxty,\apqb)$ is bounded and $\|R\tens S\|\leq \|\rlin\|\;\|S\|$.
\end{proof}

As stated in Proposition \ref{isometry} for the case of the norm $\beta$, $\mathcal{D}_{p,q}^\beta(\xxx; Y^*)$ is a closed subspace of the topological dual of $(\xtx \tens Y^{**},\apq^\beta)$. In the next theorem we show that $\mathcal{D}_{p,q}^\beta(\xxx; Y^*)$ is, indeed, a dual space.

\begin{theorem}\label{tensorialrepresentation}
Let $\xxx$, $Y$ be Banach spaces and $\beta$ be a reasonable crossnorm on $\xtx$. Then
$$(\xtxty,\alpha_{q^*,p^*}^\beta)^*=\mathcal{D}_{p,q}^\beta(\xxx; Y^*)$$
holds isometrically isomorphic.
\end{theorem}

\begin{proof}
Let $T$ in $\mathcal{D}_{p,q}^\beta(\xxx; Y^*)$. Proposition \ref{isometry} for the case of the norm $\beta$ implies that $\zeta_T$ is bounded on $(\xtx\tens Y^{**},\alpha_{q^*,p^*}^\beta)$ and $\|\zeta_T\|=D_{p,q}^\beta(T)$. Define
\begin{eqnarray*}
\fhi_T:(\xtxty,\alpha_{q^*,p^*}^\beta)&\into & \kk\\
                                                      \xxt\tens y   &\mapsto &  \lev  T\xxp  ,  y \rev.
\end{eqnarray*}
Notice that for all $\xxp$ in $\xpx$ and $y$ in $Y$ we have
$$\lev  \zeta_T  ,  \xxt\tens K_Y(y)  \rev = \lev T\xxp  ,  y  \rev= \fhi_T(\xxt\tens y).$$

Hence, (iv) of Proposition \ref{tensornorm} asserts
\begin{align*}
|\fhi_T (u)| &= |\zeta_T \circ \left( I_{X_1}\tens\dots\tens I_{X_n}\tens K_Y \right) (u)  | \\
               &\leq  \|\zeta_T\| \alpha_{q^*,p^*}^\beta \left(I_{X_1}\tens\dots\tens I_{X_n}\tens K_Y \right) (u))|\\
              &\leq  D_{p,q}^\beta(T) \alpha_{q^*,p^*}^\beta (u)
\end{align*}
for all $u$ in $\xtxty$.

On the other hand, let $\fhi$ in $(\xtxty,\alpha_{q^*,p^*}^\beta)^*$ and define
\begin{eqnarray*}
T_\fhi:\xpx &\into&        Y^*\\
           \xxp &\mapsto & T_\fhi\xxp:y\mapsto \fhi(\xxt\tens y).
\end{eqnarray*}

To see that $T$ is $\beta$-$(p,q)$-dominated let $(p_i)_{i=1}^m$, $(q_i)_{i=1}^m$ and $(y_i^{**})_{i=1}^m$ finite sequences in $\sxxb$ and $Y^{**}$ respectively and let $\eta>0$. Apply the Principle of Local Reflexivity to the spaces $G:=span\{y_i^{**}|1\leq i \leq m\}\subset Y^{**}$ and $span\{ f_{T_\fhi}(p_i) -f_{T_\fhi} (q_i) | 1\leq i \leq m \}\subset Y^*$ to find a finite dimensional subspace $F\subset Y$ and an isomorphism $\phi:G\into F$ such that $\|\phi\|\leq 1+\eta$ and
$$\lev y_i^{**}  ,   f_{T_\fhi}(p_i) -f_{T_\fhi} (q_i)\rev= \lev f_{T_\fhi}(p_i) -f_{T_\fhi} (q_i)  ,   \phi (y_i^{**})  \rev \qquad 1\leq i \leq m.$$
Then, (ii) of Proposition \ref{link} for the case $\beta$ leads us to
\begin{align*}
\|(\lev y_i^{**}  ,  f_{T_\fhi}(p_i) -f_{T_\fhi} (q_i)   \rev)_{i=1}^m\|_{r^*} &= \|(\lev  f_{T_\fhi}(p_i) -f_{T_\fhi} (q_i)  , \phi (y_i^{**})  \rev)_{i=1}^m\|_{r^*} \\
          &= \|(  \fhi( (\pimqi)\tens \phi (y_i^{**})  ) )_{i=1}^m\|_{r^*} \\
     &\leq  \|\fhi\| \|(\pimqi)_{i=1}^m\|_{p}^{w,\beta}  \|(\phi (y_i^{**}))_{i=1}^m\|_q^w  \\
     &\leq   \|\fhi\|(1+\eta) \|(\pimqi)_{i=1}^m\|_{p}^{w,\beta}  \|(y_i^{**})_{i=1}^m\|_q^w.
\end{align*}
Hence $T_\fhi$ is $\beta$-$(p,q)$-dominated and $D_{p,q}^\beta(T)\leq \|\fhi\|$.

Simple algebraic manipulations show that the mappings $T\mapsto \fhi_T$ and $\fhi\mapsto T_\fhi$ are linear and inverse of each other.
\end{proof}


\section{Final Remarks}

\subsection*{Remarks on Proposition \ref{ideal}}

\begin{remark}
In (i) we saw that every bounded multilinear form $\fhi:\xpx\into \kk$ is $\pi$-$(p,q)$-dominated for all positive integer $n$, $1\leq p,q\leq\infty$, and Banach spaces $X_i$, $1\leq i\leq n$. This does not occur in the case of $p$-dominated multilinear operators in the sense of absolutely $(\frac{p}{n}; p,\dots, p)$-summing multilinear operators. This is explicit in \cite[Lemma 5.4]{jarchow07} where it is shown that if $3\leq n$, $1\leq p< \infty$ and $X$ is an infinite dimensional Banach space, then there exists a bounded multilinear form $\fhi:X\times\dots\times X\into \kk$ which is not $p$-dominated. As a consequence, the Banach spaces $\mathcal{D}_{p,q}^\beta(\xxx;Y)$ and $\mathcal{D}_{r_1,\dots r_n}^n(\xxx;Y)$ differ in general.
\end{remark}

\begin{remark}
Another consequence of (i) is that every multilinear operator of finite type $T$, that is, of the form
\begin{small}
$$T\xxp=\sumim x_{1i}^*(x^1)\dots x_{ni}^*(x^n)y_i$$
\end{small}
where $x_{ji}^*\in X_j^*, 1\leq j\leq n, 1\leq i\leq m$ and $y_i\in Y$ is in $\mathcal{D}_{p,q}^\beta(\xxx;Y)$.
\end{remark}

\begin{remark}
Since $\mathcal{D}_{p,q}^\beta(\xxx;Y)$ is a vector space then it contains the vector space of all multilinear operators $\Txxy$ whose image is contained in a finite dimensional subspace of Y and whose linearization $\tlin:\xtxb\into Y$ is bounded.
\end{remark}

\begin{remark}
Item (ii) combined with \eqref{betalip} implies that $\|T\|\leq D_{p,q}^\beta(T)$ for all $\Txxy$ in $\mathcal{D}_{p,q}^\beta(\xxx;Y)$.
\end{remark}

\begin{remark}
The ideal property in (iii) considers compositions of the form
\begin{small}
$$ST(T_1\times\dots\times T_n):\zpz\into W,$$
\end{small}
where $T_i$, $1\leq i \leq n$ and $S:Y\into W$ are bounded linear operators and $T$ is in $\mathcal{D}_{p,q}^\beta(\xxx;Y)$. This way $ST(T_1\times\dots\times T_n)$ is $\pi$-$(p,q)$-dominated and $D_{p,q}^\pi(ST(T_1\times\dots\times T_n))\leq \|S\|D_{p,q}^\beta(T)\|T_1\|\dots\|T_n\|$.
\end{remark}

\begin{remark}
If $\beta$ and $\theta$ are two reasonable crossnorms on $\xtx$ such that $\beta\leq\lambda\theta$, for some $\lambda>0$, then the identity  $I:\mathcal{D}_{p,q}^\theta(\xxx;Y)\into\mathcal{D}_{p,q}^\beta(\xxx;Y)$ is bounded and $\|I\|\leq \lambda$. In particular, if $\beta$ and $\theta$ are equivalent, then the Banach spaces $\mathcal{D}_{p,q}^\theta(\xxx;Y)$ and $\mathcal{D}_{p,q}^\beta(\xxx;Y)$ are isomorphic.
\end{remark}

\subsection*{Remarks on Proposition \ref{tensornorm}}

\begin{remark}
As a result of (ii) and (iii) we have that $\apqb(\xxt\tens y)=\|x^1\|\dots\|x^n\|\|y\|$ for all $x^i\in X_i$, $1 \leq i \leq n$ and $y\in Y$.
\end{remark}

\begin{remark}
From (iii) we have that functionals of the form $x_1^*\tens\dots\tens x_n^*\tens y^*:(\xtxty,\apqb)\into \kk$ are bounded and $\|x_1^*\tens\dots\tens x_n^*\tens y^*\|= \|x_1^*\|\dots\|x_n^*\|\|y^*\|$.
\end{remark}

\begin{remark}
In the uniform property (iv) we are considering operators of the form
\begin{small}
\begin{eqnarray*}
R_1\tens\dots\tens R_n \tens S:(\ztztw,\alpha_{p,q}^\pi)&\into& (\xtxty,\apqb)\\
\zzt\tens w &\mapsto &  R_1(x^1)\tens\dots\tens R_n(x^n)\tens S(w).
\end{eqnarray*}
\end{small}
Moreover if $\beta$ and $\theta$ are reasonable crossnorms on $\xtx$ such that $\beta\leq\lambda\theta$, for some $\lambda>0$, then the identity $I:(\xtxty,\alpha_{p,q}^\theta)\into (\xtxty,\apqb)$ has norm $\leq\lambda$. Furthermore, if $\beta$ and $\theta$ are equivalent, then $(\xtxty,\alpha_{p,q}^\theta)$ and $(\xtxty,\apqb)$ are isomorphic.
\end{remark}

\subsection*{Chevet-Saphar Tensor Norms}

As well as in the linear case, particular cases of the Laprest\'e tensor norm $\apq^\pi$ reduce to generalizations of the Chevet-Saphar tensor norms. The reader may check \cite{chevet69, cohen73, lapreste76, saphar66} for the original references. Other good expositions are also found in \cite{defant93, ryan02}. The next proposition requires standard techniques of tensor products, so we omit the proof. A version for general reasonable crossnorms are also valid.

\begin{proposition}
Let $\xxx$ and $Y$ be Banach spaces. Then:
\begin{itemize}
   \item[i)] $d_p^\pi(u):=\alpha_{1,p}^\pi(u)=\inf\left\{  \|(\pimqi)_i \|_{p}^{w,\pi}  \|(y_i)_i\|_{p^*}   \,\Big|\,  u=\sumim (\pimqi)\tens y_i  \right\}.$
   \item[ii)] $w_p^\pi(u):=\alpha_{p,p^*}^\pi(u)=\inf\left\{  \|(\pimqi)_i \|_{p}^{w,\pi}  \|(y_i)_i\|_{p^*}^w   \,\Big|\,  u=\sumim (\pimqi)\tens y_i  \right\}.$
   \item[iii)] $g_p^\pi(u):=\alpha_{p,1}^\pi(u)=\inf\left\{  \|(\pimqi)_i \|_{p}^\pi  \|(y_i)_i\|_{p^*}^w   \,\Big|\,  u=\sumim (\pimqi)\tens y_i  \right\}.$
   \item[iv)] $\alpha_{1,1}^\pi(u)=d_{1}^\pi(u)=g_1^\pi(u)=\pi(u; \xxx, Y)$
   $$=\inf \left\{ \sumim \pi(\pimqi;\xxx) \| \|y_i\|   \,\Big|\,  u=\sumim (\pimqi)\tens y_i  \right\}$$
\end{itemize}
\end{proposition}

As it is expected, the topological duals of the normed spaces obtained from the previous proposition provide generalizations of classes of linear operators, namely, $p$-summing, $p$-dominated, bounded and operators whose dual is $p^*$-summing. According to Theorem \ref{tensorialrepresentation} we obtain the next proposition.

\begin{proposition}\label{indexes}
Let $\xxx$ and $Y$ be Banach spaces. Then
\begin{itemize}
     \item [i)]   $(\xtxty,d_p^\pi)^*=\mathcal{D}_{p^*,\infty}^\pi(\xxx;Y^*)=\Pi_{p^*}^{Lip}(\xxx;Y^*)$.
     \item [ii)]  $(\xtxty,w_p^\pi)^*=\mathcal{D}_{p,p^*}^\pi(\xxx;Y^*)$
     \item [iii)] $(\xtxty,g_p^\pi)^*=\mathcal{D}_{\infty,p^*}^\pi(\xxx;Y^*)$.
     \item [iv)] $(\xtxty, \pi)^*=\mathcal{D}_{\infty,\infty}(\xxx;Y^*)=\mathcal{L}(\xxx;Y^*)$
\end{itemize}
\end{proposition}

Item (i) in previous proposition was proved for the first time in the doctoral dissertation of J. Angulo \cite[Theorem 2.26]{angulo10}. An interesting case occurs in (iv) in the version of the norm $\beta$. In this case,
\begin{small}
$$\alpha_{1,1}^\beta(u)=\pi^\beta(u):=\inf\left\{ \sumim \beta(\pimqi) \| \|y_i\| \Big|  u=\sumim (\pimqi)\tens y_i  \right\}$$
\end{small}
for all $u$ in $\xtxty$. Even more $(\xtxty,\pi^\beta)^*$ is linearly isometric to $\mathcal{L}^\beta(\xxx;Y^*)$, the Banach space of all bounded multilinear operators with the norm $\Lb$ (see \eqref{betalip}).

\subsection*{Remarks on the Local Behavior}

In some particular cases it is not necessary to consider the restriction of the reasonable crossnorm $\beta$.

\begin{corollary}
Let $\xxx$ and $Y$ be Banach spaces and $\beta$ be and injective tensor norm. Then
\begin{itemize}
\item  [i)]  $D_{p,q}^\beta(T)=\sup \{D_{p,q}^\beta(Q_L f_T I_{\eee}) | E_i\in \mathcal{F}(X_i), F\in \mathcal{CF}(Y)\}$.
\item [ii)]  $\apqb(u; \xxx,Y)=\inf\{ \apq^{\beta}(u; \eee,F) | E_i\in \mathcal{F}(X_i), F\in \mathcal{F}(Y)\}$.
\end{itemize}
\end{corollary}

The previous corollary is also true if we replace the Banach spaces $X_i$ by Hilbert spaces $H_i$ and $\beta$ by the reasonable crossnorm $\|\cdot\|_2$ (the norm which makes $H_1\tens\dots\tens H_n$ a Hilbert space, see \cite[Section 2.6]{kadison83} for details).

According to \cite[Theorem 7.29]{ryan02} the biggest injective norm $/\pi\backslash$ is equivalent to $\omega_2$ (see \cite[Section 7.4]{ryan02}) (in the case n=2). As a consequence we have the following

\begin{corollary}
Let $\xxx$ and $Y$ be Banach spaces. Then
$$\mathcal{D}_{p,q}^{\omega_2}(X_1,X_2;Y)\cong \mathcal{D}_{p,q}^{/\pi\backslash}(X_1,X_2;Y)$$
holds isomorphic and their Banach-Mazur distance is less than the Grothendieck's constant $K_G$. In particular,
$$D_{p,q}^{\omega_2}(T)\leq D_{p,q}^{/\pi\backslash}(T)\leq K_GD_{p,q}^{\omega_2}(T)$$
for all $T$ in $\mathcal{D}_{p,q}^{\omega_2}(X_1,X_2;Y)$.
\end{corollary}

\subsection*{Proposal for $(p,q)$-dominated polynomials}

A version of $(p,q)$-domination for polynomials can also be proposed. A similar procedure was applied in \cite[Section 5]{fernandez-unzueta18b} to obtain the notion of polynomials factoring through Hilbert spaces (under the perspective of $\Sigma$-operators). Other versions of dominated polynomials can be found in \cite{botelho04, botelho06, cilia05, melendez99, rueda14}. Recall that a  mapping $p:X\rightarrow Y$ between Banach spaces is a {\sl homogeneous polynomial of degree $n$}
if there exists a multilinear mapping $T_p: X\times\ldots\times X \rightarrow Y$ such that $p(x)=T_p(x,\stackrel{n}{\ldots},x)$.

If we denote by $\pi_{n,s}$ the symmetric projective tensor norm on the symmetric tensor product $\tens^{n,s} X$ and $\tens^n x:=x\tens\dots\tens x$ we propose the next definition.

\begin{definition}\label{polynomial}
Let $1\leq p,q\leq \infty$ such that $\frac{1}{p}+\frac{1}{q}\leq 1$. Take the unique $r\in[1,\infty]$ such that $1=\frac{1}{r}+\frac{1}{p}+\frac{1}{q}$. The $n$-homogeneous polynomial $p:X\into Y$ is named $\pi_{n,s}$-$(p,q)$-dominated if there exists a constant $C>0$ such that
\begin{equation*}
 \|(\lev y_i^*  ,  p(x_i)-p(z_i)  \rev)_{i=1}^m\|_{r^*}\leq C\|(\tens^n x_i-\tens^n z_i)_{i=1}^m\|_p^{w,\pi_{n,s}}  \|(y_i^*)_{i=1}^m\|_q^w
\end{equation*}
holds for all finite sequences $(x_i)_{i=1}^m$, $(z_i)_{i=1}^m$ in $X$ and $(y_i^*)_{i=1}^m$ in $Y^*$. Define $D_{p,q}^\pi(p)$ as the infimum of all the constants $C$ as above.
\end{definition}

It is not difficult to prove that every multilinear operator $T$ in $\mathcal{D}_{p,q}^\pi(X,\dots, X; Y)$ defines an $\pi_{n,s}$-$(p,q)$-dominated n-homogeneous polynomial $p:X\into Y$ and that $D_{p,q}^{\pi_{n,s}}(p)\leq \frac{n^n}{n!} D_{p,q}^\pi(T)$. Also,  a composition $SpR$ is $\pi_{n,s}$-$(p,q)$-dominated if $p$ is and $R$ and $S$ are bounded linear operators; moreover, $D_{p,q}^{\pi_{n,s}}(RpS)\leq \|R\|D_{p,q}^{\pi_{n,s}}(p)\|S\|^n$.
\\
\\

{\bf Acknowledgments.} The first author was partially supported by CONACYT grant 284110. The second author was partially supported by CONACYT scholarship 36073.


\bibliographystyle{amsplain}

\end{document}